\documentclass[11pt,reqno]{amsart}

\usepackage{amsfonts}
\usepackage{amsmath,amscd}
\usepackage{amssymb}
\usepackage{tikz}
\usetikzlibrary{shapes.geometric, arrows}
\usepackage{tikz-cd}
\usepackage{amsthm}
\usepackage{cases}

\usepackage{galois}
\usepackage{chemarrow}
\usepackage[all]{xy}
\usepackage{graphicx}
\usepackage[colorlinks,
            linkcolor=blue,
            anchorcolor=blue,
            citecolor=red
            ]{hyperref}	
\usepackage{mathrsfs}
\usepackage{extarrows}
\usepackage{texnames}
\usepackage[nameinlink,capitalize]{cleveref}

\def\rank{{\textrm{rank}}}

\def\End{{\rm End}}

\def\Hom{{\rm Hom}}

\def\Ker{{\rm Ker}}

\def\rank{{\rm rank}}

\theoremstyle{plain}
\newtheorem{introtheorem}{Theorem}

\newtheorem{introconjecture}[introtheorem]{Conjecture}

\newtheorem{theorem}{Theorem}[section]

\newtheorem{proposition/example}[theorem]{Proposition/Example}
\newtheorem{proposition}[theorem]{Proposition}
\newtheorem{corollary}[theorem]{Corollary}
\newtheorem{lemma}[theorem]{Lemma}

\theoremstyle{definition}
\newtheorem{definition}[theorem]{Definition}

\newtheorem{example}[theorem]{Example}

\newtheorem*{rmk}{Remark}

\newtheorem{conjecture/question}[theorem]{Conjecture/Question}

\newtheorem{remark/definition}[theorem]{Remark/Definition}
\newtheorem{definition/notation}[theorem]{Definition/Notation}

\numberwithin{equation}{section}

\begin{document}
\title{\textbf{Simpson Filtration and Oper Stratum Conjecture}}

\author{Zhi Hu}

\address{ \textsc{School of Science, Nanjing University of Science and Technology, Nanjing 210094,  China}\endgraf \textsc{Department of Mathematics, Mainz University, 55128 Mainz, Germany}}

\email{huz@uni-mainz.de; halfask@mail.ustc.edu.cn}

\author{Pengfei Huang}

\address{ \textsc{School of Mathematical Sciences, University of Science and Technology of China, Hefei 230026, China}\endgraf
\textsc{Laboratoire J.A. Dieudonn\'e, Universit\'e C\^ote d'Azur, CNRS, 06108 Nice, France}}

\email{pfhwang@mail.ustc.edu.cn; pfhwang@unice.fr}

\subjclass[2010]{14D20, 14D22, 14J60, 57N80}

\keywords{$\lambda$-flat bundle, Chain, System of Hodge bundles, Simpson filtration, Oper stratification, Complex variations of Hodge structure, Moduli spaces}
\date{}

\begin{abstract}
In this paper, we prove that for the oper stratification of the de Rham moduli space $M_{\mathrm{dR}}(X,r)$,  the closed oper stratum is the unique minimal stratum with dimension $r^2(g-1)+g+1$, and the open dense  stratum consisting of irreducible flat bundles with stable underlying vector bundles is the unique maximal stratum.
\end{abstract}

\maketitle
\tableofcontents

\section{Introduction}
Many moduli spaces admit stratifications, even arising from different points of view.   For example, Atiyah and Bott studied  the Morse stratification on an infinite dimensional affine space $\mathcal{A}$ of unitary connections on a $C^\infty$-vector bundle over a Riemann surface $X$ of genus $g\geq 2$, which is defined by  the gradient flow of the norm square of  moment map with respect to the gauge group \cite{AB}, and coincides with a stratification defined by the Harder--Narasimhan types of algebraic vector bundles over $X$ \cite{D}. Hitchin introduced the moduli space of (stable) Higgs bundle over $X$ \cite{hi}, it admits two stratifications: one is also the Harder--Narasimhan (HN) stratification defined  by the  Harder--Narasimhan types of the underlying vector bundles, the other one is the Bialynicki-Birula (BB) stratification defined by the fixed points of $\mathbb{C}^*$-action on the moduli space. The latter one also has the Morse theoretic interpretation. Indeed, the norm square of Higgs field is  a Morse function on the moduli space, and is a moment map with respect to the Hamiltonian $S^1$-action, by Kirwan's result \cite{K}, the stratification defined by the upwards Morse flow of that Morse function coincides with the BB stratification. In general, the HN stratification does not coincide with the BB stratification, however, they are the same for the case of rank two due to Hausel \cite{Ha}.  For the moduli space of flat bundles over $X$, the HN stratification can be obviously well-defined. In \cite{CS6}, Simpson constructed a stratification of this moduli space by embedding it into the moduli space of $\lambda$-flat bundles  (varying $\lambda\in \mathbb{C}$) and showing the natural $\mathbb{C}^*$-action on the bigger moduli space defines a BB stratification. Simpson called this stratification the \emph{oper stratification} since the ``minimal'' stratum is the moduli space of opers. So far, we do not know whether the oper stratification can be viewed as a certain Morse stratification.

Let $\mathbb{M}_{\mathrm{Hod}}(X,r)$ be the coarse moduli space of semistable $\lambda$-flat   bundles over $X$ of fixed rank $r$ with vanishing first Chern class, which is called the \emph{Hodge moduli space}. This space has a natural  fibration $\mathbb{M}_{\mathrm{Hod}}(X,r)\to\mathbb{C}, (E,D^\lambda,\lambda)\mapsto\lambda$ such that the fiber over 0 is the \emph{Dolbeault moduli space} $\mathbb{M}_{\mathrm{Dol}}(X,r)$, namely the moduli space of semistable Higgs bundles over $X$ of rank $r$ with vanishing first Chern class; and the fiber over 1 is the \emph{de Rham moduli space} $\mathbb{M}_{\mathrm{dR}}(X,r)$, namely the moduli space of  flat bundles over $X$ of rank $r$. The natural $\mathbb{C}^*$-action on $\mathbb{M}_{\mathrm{Dol}}(X,r)$ by rescaling the Higgs field can be extended to an action on $\mathbb{M}_{\mathrm{Hod}}(X,r)$ via $t\cdot(E,D^\lambda,\lambda):=(E,tD^\lambda,t\lambda)$ for $t\in\mathbb{C}^*$.
  Simpson showed in \cite{CS6} that for each $(E, D^\lambda,\lambda)\in \mathbb{M}_{\mathrm{Hod}}(X,r)$, the limit  $\lim\limits_{t\rightarrow 0}t\cdot(E,D^\lambda,\lambda)$ exists uniquely as  a $\mathbb{C}^*$-fixed point lying in $\mathbb{M}_{\mathrm{Dol}}(X,r)$. If we divide the set $V(X,r)$ of $\mathbb{C}^*$-fixed points   into the connected components $V(X,r)=\coprod_\alpha V_\alpha$, then we will have a stratification $\mathbb{M}_{\mathrm{Hod}}(X,r)=\coprod_\alpha G_\alpha$, where the locally closed subset $G_\alpha$ is defined as $G_\alpha=\{(E,D^\lambda,\lambda)\in\mathbb{M}_{\mathrm{Hod}}(X,r): \lim\limits_{t\rightarrow 0}t\cdot(E,D^\lambda,\lambda)\in V_\alpha\}$.  In particular, the restriction to  the fiber over 0 leads to the BB stratification of  $\mathbb{M}_{\mathrm{Dol}}(X,r)$, and the restriction to  the fiber over 1 gives rises to the oper stratification: $\mathbb{M}_{\mathrm{dR}}(X,r)=\coprod_{\alpha}\mathbb{S}_\alpha$, or $M_{\mathrm{dR}}(X,r)=\coprod_{\alpha}S_\alpha$ for smooth locus. 

In order to obtain the $\mathbb{C}^*$-limit points of flat bundles, the \emph{Simpson filtration} plays a crucial role, which is a filtration of subbundles satisfying two extra conditions: Griffiths transversality for the flat connection and semistability for the associated graded Higgs bundle. Such filtration always exists by iterated destabilizing modifications.

  For the oper stratification, Simpson proposed many interesting conjectures, part of them were already solved by Simpson himself and others, but many of them are still open questions.

 \begin{introconjecture} [\textbf{Foliation  Conjecture}] 
  The Lagrangian fibers of the projections $\mathbb{S}_\alpha\rightarrow V_\alpha$ fit together into a smooth Lagrangian foliation with closed leaves.
\end{introconjecture}

\begin{introconjecture}[\textbf{Nestedness Conjecture}]
The oper stratification $\mathbb{M}_{\mathrm{dR}}(X,r)=\coprod_{\alpha}\mathbb{S}_\alpha$ is nested, namely, there is a partial order on the index set $\{\alpha\}$ such that 
$$
\overline{\mathbb{S}_\alpha}=\coprod_{\beta\leq \alpha}\mathbb{S}_\beta.
$$
\end{introconjecture}

\begin{introconjecture}[\textbf{Oper Stratum Conjecture}]
The oper stratum is the unique closed stratum and the unique stratum of minimal dimension.
\end{introconjecture}

\begin{rmk}
For the first conjecture, the part of Lagrangian property is already known \cite{C,CS6} (see also Lemma \ref{halfdim}), the closedness of fibers is quite obvious, since these fibers are affine spaces by applying BB theory to Hodge moduli space (see Proposition 2.1 of \cite{H-H}, or Corollary 1.5 of \cite{C}) and the de Rham moduli space is an affine analytic variety. The whole foliation conjecture for the moduli space of rank two connections over four punctured projective line was proved by the authors in \cite{LSS}. As pointed out by Simpson, if this conjecture is affirmative, it could be useful for the context of geometric Langlands, where the full moduli stack of vector bundles might be replaced by the algebraic space    of leaves of foliations. What we should mention is that the analogous  statement of closedness on the side of the moduli space of Higgs bundles is not true. For example, one picks up  $u=(E_0,0)$ as a stable vector bundle $E_0$ of rank $r$ with trivial Higgs field, then $Y^0_u:=\{(E,\theta)\in M_{\mathrm{Dol}}(X,r): \lim\limits_{t\rightarrow 0}t\cdot(E,\theta)=u\}\simeq H^0(X,\End(E_0)\otimes \Omega^1_X)$, it is closed in $ M_{\mathrm{Dol}}(X,r)$  if and only if $E_0$ is very stable, i.e. there is  no non-zero nilpotent Higgs field on $E_0$ \cite{PP}. For the second conjecture, we only know very few, in \cite{CS6}, Simpson showed the nestedness conjecture for the case of rank two by a beautiful deformation theory argument, which in particular, implies the oper stratum conjecture. 
\end{rmk}

In this paper, we mainly focus on  the oper stratum conjecture. In particular, we will prove the following theorem which partially confirms Simpson's oper stratum conjecture (the part on dimension).

\begin{introtheorem}[= Corollary \ref{mainthm}] For the oper stratification $M_{\mathrm{dR}}(X,r)=\coprod_\alpha S_\alpha$, we have
\begin{enumerate}
  \item the open dense stratum $N(X,r)$ consisting of irreducible flat bundles such that the underlying vector bundles are stable is the unique maximal stratum with dimension $2r^2(g-1)+2$,
  \item the closed oper stratum $S_{\mathrm{oper}}$ is the unique minimal stratum with dimension $r^2(g-1)+g+1$.
\end{enumerate}
\end{introtheorem}

The paper is organized as follows. In the following section, we collect some basic materials on  the theory of $\lambda$-flat bundles, and holomorphic chains that can be used to describe the $\mathbb{C}^*$-fixed points (i.e. $\mathbb{C}$-VHSs).  In the third section, we introduce the Simpson filtrations of flat bundles and the oper stratification of de Rham moduli space, we also give an explicit description of the Simpson filtrations on flat bundles of rank 3, as well as an upper bound on the degree of subbundles. In the last section, we will give the proof of our main theorem, which reduces to the study of the connected components $V_\alpha$ since each stratum $S_\alpha$ is a fibration over $V_\alpha$ with fibers as Lagrangian submanifolds of $M_{\mathrm{dR}}(X,r)$.

\bigskip

\noindent\textbf{Acknowledgements}. 
The authors would like to express their deep gratitude to Prof. Brian Collier, Prof. Peter Gothen, Prof. Jochen Heinloth  and Prof. Richard Wentworth for communications on various occasions, and  to the anonymous referee for many valuable suggestions. The author P. Huang would like to thank Prof. Carlos Simpson for  kind help and useful discussions, and thank  Prof. Jiayu Li for his continuous encouragement.

\section{Preliminaries}

\subsection{Flat $\lambda$-Connections}

The notion of flat $\lambda$-connection as the interpolation of usual flat connection and Higgs field was suggested by Deligne \cite{D}, illustrated by Simpson in \cite{CS4} and further  studied in \cite{CS5,CS6}. We recall the definition here.

Throughout the paper, $X$ is always assumed to be a compact Riemann surface of genus $g\geq 2$, and let $K_X=\Omega^1_X$ be the canonical line bundle over $X$.

\begin{definition} [\cite{CS4}] 
Assume $\lambda\in\mathbb{C}$.
\begin{enumerate}
  \item Let $E$ be a holomorphic vector bundle over $X$, a \emph{$\lambda$-connection} on $E$ is a $\mathbb{C}$-linear operator $D^\lambda: E\to E\otimes\Omega_X^{1}$ that satisfies the following $\lambda$-twisted Leibniz rule:
$$
D^\lambda(fs)=fD^\lambda s+\lambda s\otimes  df,
$$
where $f$ and $s$ are holomorphic sections of $\mathcal{O}_X$ and $E$, respectively. If $D^\lambda\circ D^\lambda=0$ under the natural extension $D^\lambda: E\otimes\Omega_X^{p}\to E\otimes\Omega_X^{p+1}$ for any integer $p\geq0$, we call $D^\lambda$ a \emph{flat $\lambda$-connection}, and the pair $(E,D^\lambda)$ is called a  \emph{$\lambda$-flat bundle}.
\item A  $\lambda$-flat bundle $(E,D^\lambda)$ over $X$ is called \emph{stable} (resp. \emph{semistable}) if for any $\lambda$-flat subbundle $(F,D^\lambda|_F)$ of $0<\rank(F)<\rank(E)$, we have the following inequality
$$
\mu(F)< (\text{resp.} \leq) \,\mu(F),
$$
where $\mu(\bullet)=\frac{\deg(\bullet)}{\rank(\bullet)}$ denotes the \emph{slope} of bundle. And we call $(E,D^\lambda)$ is \emph{polystable} if it decomposes as a direct sum of stable $\lambda$-flat bundles with the same slope.
\end{enumerate}
\end{definition}

Let $\mathcal{M}_{\mathrm{Hod}}(X,r)$ be the moduli stack of
$\lambda$-flat bundles (varying $\lambda\in\mathbb{C}$) over $X$ of rank $r$ with vanishing first Chern class,  and let  $\mathbb{M}_{\mathrm{Hod}}(X,r)$ be the coarse moduli space of semistable $\lambda$-flat bundles of this stack, called the \emph{Hodge moduli space}. It's known that  $\mathbb{M}_{\mathrm{Hod}}(X,r)$ is a quasi-projective variety and parameterizes the isomorphism classes of  polystable $\lambda$-flat bundles \cite{CS4}, let $M_{\mathrm{Hod}}(X,r)$ be the  smooth locus of $\mathbb{M}_{\mathrm{Hod}}(X,r)$, which is  a Zariski dense open subset  and parameterizes the isomorphism classes of  stable $\lambda$-flat bundles. There is a natural fibration $\pi: \mathbb{M}_{\mathrm{Hod}}(X,r)\rightarrow\mathbb{C}, (E,D^\lambda,\lambda)\mapsto\lambda$ such that the fiber $\pi^{-1}(\lambda)=:\mathbb{M}_{\mathrm{Hod}}^\lambda(X,r)$ is the coarse moduli space of semistable $\lambda$-flat bundles over $X$ of rank $r$ with vanishing first Chern class, in particular, 
\begin{itemize}
\item[$\bullet$] $\pi^{-1}(1)=\mathbb{M}_{\mathrm{dR}}(X,r)$, the coarse moduli space of flat bundles over $X$ of rank $r$, called the \emph{de Rham moduli space}, which is algebraically isomorphic to $\mathbb{M}_{\mathrm{dR}}(X,r)$ for any $\lambda\neq0$;
\item[$\bullet$] $\pi^{-1}(0)=\mathbb{M}_{\mathrm{Dol}}(X,r)$,  the coarse moduli space of semistable Higgs bundles over $X$ of rank $r$ with vanishing first Chern class, called the \emph{Dolbeault moduli space}.
\end{itemize}

The natural $\mathbb{C}^*$-action on $\mathbb{M}_{\mathrm{Dol}}(X,r)$ via $t\cdot(E,\theta):=(E,t\theta)$ can be generalized to the $\mathbb{C}^*$-action on $\mathbb{M}_{\mathrm{Hod}}(X,r)$ via $t\cdot (E,D^\lambda,\lambda) := (E,tD^\lambda,t\lambda)$. As this action change the fibers over $\lambda\neq0$, therefore, all the fixed points must lie in the fiber over 0, that is, in $\mathbb{M}_{\mathrm{Dol}}(X,r)$. It is known that a such $\mathbb{C}^*$-fixed point is a polystable Higgs bundle that has a structure of system of Hodge bundles (\cite[Lemma 4.1]{CS3}), along the phraseology of \cite{C}, we will call it a complex variation of Hodge structure (briefly as $\mathbb{C}$-VHS). For the Dolbeault moduli space $\mathbb{M}_{\mathrm{Dol}}(X,r)$, the properness and the $\mathbb{C}^*$-equivariant property of the Hitchin map preserve that, for each Higgs bundle $(E,\theta)$, the limit point $\lim\limits_{t\to0}t\cdot(E,\theta)$  exists as a $\mathbb{C}$-VHS.  There is no analogue of Hitchin map for the Hodge moduli space $\mathbb{M}_{\mathrm{Hod}}(X,r)$, however in \cite{CS6}, Simpson showed that for the $\mathbb{C}^*$-action on $\mathbb{M}_{\mathrm{Hod}}(X,r)$, the limit point $\lim\limits_{t\rightarrow 0}t\cdot (E,D^\lambda,\lambda)$ also exists as a $\mathbb{C}$-VHS (see Section \ref{sec3}).

\subsection{Chains and  $\mathbb{C}$-VHSs}

A $\mathbb{C}$-VHS has the form $(E=\bigoplus_{i=1}^lE_i,\theta=\bigoplus_{i=1}^{l-1}\theta_i)$ with each $\theta_i: E_i\to E_{i+1}\otimes K_X$, $\rank(E_i)=r_i$, and $\deg(E_i)=d_i$. The pair $(\overrightarrow{r},\overrightarrow{d}):=(r_1,\cdots,r_l; d_1,\cdots,d_l)$ is called the \emph{type} of the $\mathbb{C}$-VHS. This can be characterized by a chain of holomorphic bundles with certain stability parameter.

 \begin{definition}[\cite{A,B}] 
 \mbox{}
  \begin{enumerate}
   \item  Let $\overrightarrow{r}=(r_1,\cdots, r_l), \overrightarrow{d}=(d_1,\cdots,d_l)$ and $|\overrightarrow{r}|=\sum_{i=1}^lr_i, |\overrightarrow{d}|=\sum_{i=1}^ld_i$. A \emph{chain} $Ch_{\overrightarrow{r},\overrightarrow{d}}$ of length $l$ is a tuple $$(\mathcal{E}_i,i=1,\cdots,l; \varphi_i, i=1,\cdots i-1),$$ consisting of holomorphic bundles $\mathcal{E}_i$ on $X$ with $\rank(\mathcal{E}_i)=r_i, \deg(\mathcal{E}_i)=d_i (i=1,\cdots,l)$, and holomorphic morphisms $\varphi_i:\mathcal{E}_i\rightarrow \mathcal{E}_{i+1} (i=1,\cdots,l-1)$. We write a chain as
  $$
  Ch_{\overrightarrow{r},\overrightarrow{d}}:\mathcal{E}_1\xrightarrow{\varphi_1}\mathcal{E}_2\xrightarrow{\varphi_2}\cdots\xrightarrow{\varphi_{l-1}}\mathcal{E}_l,
  $$ 
  the pair $(\overrightarrow{r},\overrightarrow{d})$ is called the \emph{type} of the chain. If each $\varphi_i$ dose not vanish, the chain is called \emph{indecomposable}.
   \item Let $\overrightarrow{{\alpha}}=(\alpha_1,\cdots,\alpha_l)\in\mathbb{R}^l$, and call it a \emph{stability parameter}. For a chain $Ch_{\overrightarrow{r},\overrightarrow{d}}$, we introduce its \emph{$\overrightarrow{\alpha}$-slope} as
 $$
 \mu_{\overrightarrow{\alpha}}(Ch_{\overrightarrow{r},\overrightarrow{d}})=\frac{|\overrightarrow{d}|+\sum_{i=1}^l\alpha_ir_i}{|\overrightarrow{r}|}.
 $$
 A chain $Ch_{\overrightarrow{r},\overrightarrow{d}}$ is called \emph{$\overrightarrow{\alpha}$-stable} (resp., \emph{$\overrightarrow{\alpha}$-semistable}) if for any non-zero proper subchain $Ch_{\overrightarrow{r^\prime},\overrightarrow{d^\prime}}$ we have
 $$
 \mu_{\overrightarrow{\alpha}}(Ch_{\overrightarrow{r^\prime},\overrightarrow{d^\prime}})<(\mathrm{resp.},\ \leq)\ \mu_{\overrightarrow{\alpha}}(Ch_{\overrightarrow{r},\overrightarrow{d}}).
 $$
  A chain $Ch_{\overrightarrow{r},\overrightarrow{d}}$ is called \emph{$\overrightarrow{\alpha}$-polystable} if it can be decomposed into the direct sum of two subchains $Ch_{\overrightarrow{r},\overrightarrow{d}}=Ch_{\overrightarrow{r_1},\overrightarrow{d_1}}\bigoplus Ch_{\overrightarrow{r_2},\overrightarrow{d_2}}$ such that $Ch_{\overrightarrow{r_i},\overrightarrow{d_i}} (i=1,2)$ is $\overrightarrow{\alpha_i}$-stable with $\mu_{\overrightarrow{\alpha_i}}(Ch_{\overrightarrow{r_i},\overrightarrow{d_i}})=\mu_{\overrightarrow{\alpha}}(Ch_{\overrightarrow{r},\overrightarrow{d}})$, where $\overrightarrow{\alpha_i}=\overrightarrow{\alpha}|_{Ch_{\overrightarrow{r_i},\overrightarrow{d_i}}}$.
 \end{enumerate}
 \end{definition}

In \cite{PH}, the authors considered the necessary conditions for the existence of $\overrightarrow{\alpha}$-semistable chains:

\begin{proposition} [\cite{PH}]\label{1}
Let $Ch_{\overrightarrow{r},\overrightarrow{d}}$ be a  $\overrightarrow{{\alpha}}$-semistable chain of length $l$, where $\overrightarrow{{\alpha}}=(\alpha_1,\cdots,\alpha_l)$ is a  stability parameter satisfying $\alpha_1>\cdots>\alpha_l$, and let $\mu=\mu_{\overrightarrow{{\alpha}}}(Ch_{\overrightarrow{r},\overrightarrow{d}})$, then
\begin{enumerate}
  \item for all $j\in{2,\cdots,l}$, we have$$\frac{\sum_{i=j}^l(d_i+\alpha_i r_i)}{\sum_{i=j}^lr_i}\leq \mu;$$
  \item for all $j$ such that $r_j=r_{j+1}$, we have $$d_j\leq d_{j+1};$$
  \item for all $1\leq k<j\leq l$ such that $r_k<\min\{r_{k+1},\cdots, r_j\}$, we have
  $$
  \frac{\sum_{i\notin[k,j]}(d_i+\alpha_ir_i)+(j-k+1)d_k+(\sum_{i=k}^j\alpha_i)r_k}{\sum_{i\notin[k,j]}r_i+(j-k+1)r_k}\leq \mu;
  $$
  \item for all $1\leq k<j\leq l$ such that $r_k>\max\{r_{k+1},\cdots, r_j\}$, we have
  $$
  \frac{\sum_{i=k}^{j-1}(d_i-d_j+\alpha_i(r_i-r_j))}{\sum_{i=k}^{j-1}(r_i-r_j)}\leq \mu.
  $$
\end{enumerate}
\end{proposition}

Given a $\mathbb{C}$-VHS of the form $(E=\bigoplus_{i=1}^lE_i,\theta=\bigoplus_{i=1}^{l-1}\theta_i)$ of type $(\overrightarrow{r},\overrightarrow{d}):=(r_1,\cdots,r_l; d_1,\cdots,d_l)$, we can obtain a chain with a parameter $\delta\in\mathbb{Z}$ as
 $$
 Ch_{\overrightarrow{r},\overrightarrow{d'}}:\mathcal{{E}}_1\xrightarrow{\varphi_1}\mathcal{{E}}_2\xrightarrow{\varphi_2}\cdots\xrightarrow{\varphi_{l-1}}\mathcal{{E}}_l,
 $$
 where each $\mathcal{{E}}_i={E}_i\otimes K_X^{-(l-i+\delta)}$, $\varphi_i=\theta_i\otimes \mathrm{Id}$ and $d'_i=d_i-r_i(l-i+\delta)(2g-2)$. We assign each $\mathcal{E}_i$ an integer 
 $$
 \alpha_i=(l-i+\delta)(2g-2)
 $$ 
 to form a  stability parameter $\overrightarrow{{\alpha}}_{\mathrm{Higgs}}$, then $Ch_{\overrightarrow{r},\overrightarrow{d'}}$ is 
 $\overrightarrow{{\alpha}}_{\mathrm{Higgs}}$-semistable.

\begin{proposition}\label{11}
Denote by $V(X,r)$ the set of all $\mathbb{C}$-VHSs lying in $\mathbb{M}_{\mathrm{Dol}}(X,r)$. 
Let $V^{\overrightarrow{r},\overrightarrow{d}}\subset V(X,r)$ be the subset that consists of $\mathbb{C}$-VHSs of type $(\overrightarrow{r},\overrightarrow{d})$ with each component $\theta_i$ of the Higgs field $\theta$ non-zero. If $V^{\overrightarrow{r},\overrightarrow{d}}$ is non-empty,  then $(\overrightarrow{r},\overrightarrow{d})$ should satisfy the following conditions:
\begin{enumerate}
  \item for all $1<j\leq l$, we have  
  \begin{align}\label{2.1}
  \sum_{i=j}^ld_i<0,
  \end{align}
  \item for all $j$ such that $r_j=r_{j+1}$, we have
     \begin{align}\label{2.2}
     \frac{d_j}{r_j}-\frac{d_{j+1}}{r_{j+1}}\leq 2g-2,
     \end{align}  
  \item for all $1\leq k<j\leq l$ such that $r_k<\min\{r_{k+1},\cdots, r_j\}$, we have
  \begin{align}\label{2.4.3}
  -\sum_{i=k+1}^jd_i+(j-k)(d_k-(j-k+1)(g-1)r_k)\leq0,
  \end{align}
  \item for all $1\leq k<j\leq l$ such that $r_k>\max\{r_{k+1},\cdots, r_j\}$, we have
  \begin{align}\label{2.4}
  \sum_{i=k}^{j-1}d_i-(j-k)(d_j+(j-k+1)(g-1)r_j)\leq 0,
  \end{align}
\end{enumerate}
Conversely, if the type $(\overrightarrow{r},\overrightarrow{d})$ satisfies the conditions \eqref{2.1}-\eqref{2.4}, then $V^{\overrightarrow{r},\overrightarrow{d}}$ is non-empty.
\end{proposition}

\begin{proof}
Viewing a $\mathbb{C}$-VHS as a chain of certain type with the stability parameter $\overrightarrow{\alpha}_{\mathrm{Higgs}}=(\alpha_1,\cdots,\alpha_l)$, since $\alpha_1>\cdots>\alpha_l$, then we can apply Proposition \ref{1} to obtain the four inequalities in proposition.

Conversely, given $(\overrightarrow{r},\overrightarrow{d})$, by a result of \cite{BGG}, the conditions \eqref{2.1}-\eqref{2.4} implies that there is  an $\overrightarrow{\alpha}_{\mathrm{Higgs}}$-semistable chain $Ch_{\overrightarrow{r},\overrightarrow{d_0}}$ of type $(\overrightarrow{r},\overrightarrow{d_0}=\overrightarrow{d}-\overrightarrow{\alpha}_{\mathrm{Higgs}})$ for $\overrightarrow{\alpha}_{\mathrm{Higgs}}=((l-1)(2g-2),\cdots,2g-2,0)$, hence produces a $\mathbb{C}$-VHS via the suitable semisimplification of the Jordan--H\"older filtration. Moreover, the condition \eqref{2.1} guarantees such chain is indecomposable, namely each $\theta_i$ is non-zero, thus the last statement is shown.
\end{proof}

\section{Simpson Filtration and Oper Stratification}\label{sec3}

\begin{definition} [\cite{CS6}] 
Let $E$ be a holomorphic vector bundle over $X$ with a holomorphic flat connection $\nabla:E\rightarrow E\otimes_{\mathcal{O}_X}K_X$. A decreasing filtration $\mathcal{F}=\{F^\bullet\}$ of $E$ by strict subbundles
$$
E=F^0\supset F^1\supset\cdots\supset F^k=0
$$
is called a \emph{Simpson filtration} if it satisfies the following two conditions:
\begin{itemize}
  \item Griffiths transversality: $\nabla: F^p\rightarrow F^{p-1}\otimes_{\mathcal{O}_X}K_X$,
  \item graded-semistability: the associated graded Higgs bundle $(\mathrm{Gr}_{\mathcal{F}}(E),\mathrm{Gr}_{\mathcal{F}}(\nabla))$, where $\mathrm{Gr}_{\mathcal{F}}(E)=\bigoplus_pE^p$ with $E^p=F^p/F^{p+1}$ and $\mathrm{Gr}_{\mathcal{F}}(\nabla)=\bigoplus_p\theta^p$ with $\theta^p: E^p\rightarrow E^{p-1}\otimes_{\mathcal{O}_X}K_X$ induced from $\nabla$, is a semistable Higgs bundle.
\end{itemize}
\end{definition}

In \cite{CS6}, Simpson studied the $\mathbb{C}^*$-action on flat bundles and obtained the following nice theorem:

\begin{theorem} [\cite{CS6}]\label{22} 
Let $(E,\nabla)$ be a flat bundle over  $X$.
\begin{enumerate}
\item There exist  Simpson filtrations $\mathcal{F}$ on $(E,\nabla)$.
\item Let $\mathcal{F}_1$, $\mathcal{F}_2$ be two Simpson filtrations on $(E,\nabla)$, then the associated graded Higgs bundles $(\mathrm{Gr}_{\mathcal{F}_1}(E),\mathrm{Gr}_{\mathcal{F}_1}(\nabla))$ and $(\mathrm{Gr}_{\mathcal{F}_2}(E),\mathrm{Gr}_{\mathcal{F}_2}(\nabla))$ are $S$-equivalent\footnote{We say two semistable Higgs bundles are \emph{$S$-equivalent} if their associated graded Higgs bundles defined by the Jordan-H\"older filtrations are isomorphic.}.
\item $(\mathrm{Gr}_\mathcal{F}(E),\mathrm{Gr}_\mathcal{F}(\nabla))$ is a stable Higgs bundle iff the Simpson filtration is unique.
\item $\lim\limits_{t\rightarrow 0}t\cdot(E,\nabla)=(\mathrm{Gr}_\mathcal{F}(E),\mathrm{Gr}_\mathcal{F}(\nabla))$.
\end{enumerate}
\end{theorem}

The existence of Simpson filtrations is obtained via a beautiful iterating process, here we sketch how it works.

Suppose $(E,\nabla)$ admits a filtration 
$$
\mathcal{F}: E=F^0\supset F^1\supset\cdots\supset F^k=0
$$
that satisfies the Griffiths transversality $\nabla(F^p)\subset F^{p-1}\otimes K_X$, and such that the associated Higgs bundle $(V,\theta):=(\mathrm{Gr}_{\mathcal{F}}(E),\mathrm{Gr}_{\mathcal{F}}(\nabla))$ is not semistable. To see the existence of such filtration, we can begin with the trivial filtration $E=F^0\supset F^1=0$,  the graded Higgs bundle will be $(\mathrm{Gr}_{\mathcal{F}}(E),\mathrm{Gr}_{\mathcal{F}}(\nabla))=(E,0)$.  Take $H\subset(V,\theta)$ to be the maximal destabilizing subsheaf, which is known being unique and a subbundle of $V$, and the quotient $V/H$ is also a subbundle of $E$. As a sub-Higgs bundle of a $\mathbb{C}$-VHS, $H$ is also  a $\mathbb{C}$-VHS, and is a sub-$\mathbb{C}$-VHS of $(V,\theta)$, that is, $H=\bigoplus H^p$ with each $H^p=H\bigcap \mathrm{Gr}_{\mathcal{F}}^p(E)\subset F^p/F^{p+1}$ being a strict subbundle.

The new filtration $\mathcal{G}=\{G^\bullet\}$ is defined as
$$
G^p:=\Ker\left(E\to\frac{E/F^p}{H^{p-1}}\right).
$$
It satisfies the Griffiths traversality since $\theta(H^p)\subset H^{p-1}\otimes K_X$, and it fits into the short exact sequence
$$
0\longrightarrow \mathrm{Gr}_{\mathcal{F}}^p(E)/H^p\longrightarrow\mathrm{Gr}_{\mathcal{G}}^p(E)\longrightarrow H^{p-1}\longrightarrow0.
$$
If the new resulting graded Higgs bundle $(\mathrm{Gr}_{\mathcal{G}}(E),\mathrm{Gr}_{\mathcal{G}}(\nabla))$ is still not semistable, then we continue this process to obtain a new graded Higgs bundle. By introducing three bounded invariants, Simpson showed that the iteration process will strictly decrease these invariants in lexicographic order\footnote{Details on these invariants can be found in Simpson's paper \cite{CS6} (see also the second named author's thesis \cite{Hua}). }. Therefore, after a finite step, we will find a filtration such that the associated graded Higgs bundle is semistable. 

This can be concluded as the following algorithm flowchart:

\medskip

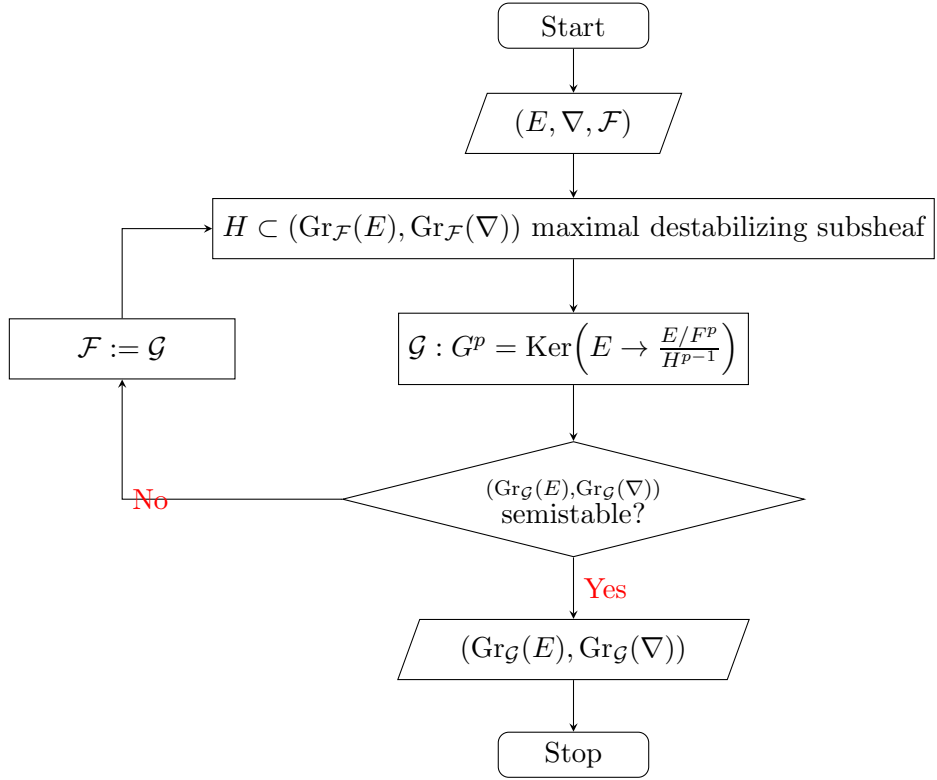
\tikzstyle{startstop} = [rectangle, rounded corners, minimum width = 2cm, minimum height=0.6cm,text centered, draw = black]
\tikzstyle{io} = [trapezium, trapezium left angle=70, trapezium right angle=110, minimum width=2cm, minimum height=0.8cm, text centered, draw=black]
\tikzstyle{process} = [rectangle, minimum width=3cm, minimum height=0.8cm, text centered, draw=black]
\tikzstyle{decision} = [diamond, aspect = 4, text centered, draw=black]
\tikzstyle{arrow} = [->,>=stealth]
\vspace*{-15pt}
\[
\begin{tikzpicture}[node distance=0.8cm]
\node[startstop](start){Start};
\node[io, below of = start, yshift = -0.5cm](in1){$(E,\nabla,\mathcal{F})$};
\node[process, below of = in1, yshift = -0.6cm](dec1){$H\subset(\mathrm{Gr}_{\mathcal{F}}(E),\mathrm{Gr}_{\mathcal{F}}(\nabla))$ maximal destabilizing subsheaf};
\node[process, below of = dec1, yshift = -0.8cm](dec2){$\mathcal{G}: G^p=\mathrm{Ker}\Big(E\to \frac{E/F^p}{H^{p-1}}\Big)$};
\node[decision, below of = dec2, yshift = -1.2cm](pro1){ $\overset{(\mathrm{Gr}_{\mathcal{G}}(E),\mathrm{Gr}_{\mathcal{G}}(\nabla))}{\mathrm{semistable} ?}$};
\node[process, below of = dec1, yshift = -0.8cm,xshift=-6cm](dec3){$\mathcal{F}:=\mathcal{G}$};
\node[io, below of = pro1, yshift = -1.2cm](out1){$(\mathrm{Gr}_{\mathcal{G}}(E),\mathrm{Gr}_{\mathcal{G}}(\nabla))$};
\node[startstop, below of = out1, yshift = -0.6cm](stop){Stop};
\draw [arrow] (start) -- (in1);
\draw [arrow] (in1) -- (dec1);
\draw [arrow] (dec1) -- (dec2);
\draw [arrow] (dec2) -- (pro1);
\draw [arrow] (pro1) -| node [right] {\color{red}{No}} (dec3);
\draw [arrow] (dec3) |- (dec1);
\draw [arrow] (pro1) -- node [right] {\color{red}{Yes}} (out1);
\draw [arrow] (out1) -- (stop);
\end{tikzpicture}
\]

Generally speaking, calculating the Simpson filtration for a given flat bundle is quite hard. For rank 2 case, the Simpson filtration is exactly given by the Harder--Narasimhan filtration of the bundle itself. But for the flat bundles of higher rank, it's very complicated. Here we give an explicit description for rank 3 case (details especially the proof can be found in \cite{Hua1}). An analogous result for the Higgs bundles of rank 3 can be found in \cite{GZ}.

\begin{example}
Let $(E,\nabla)$ be a   flat bundle of rank 3 over a  $X$, and we assume $E$ is not a stable bundle. The Simpson filtration on $(E,\nabla)$ is described as follows. 
\begin{enumerate}
  \item Assume the Harder-Narasimhan filtration on $E$ is given by $H^1\subset E$ with $\rank (H^1)=1, \deg(H^1)=d$, then $0<d\leq \frac{2}{3}(2g-2)$.  The flat connection $\nabla$ induces a nonzero morphism
  $\theta: H^1 \rightarrow E/H^1\otimes K_X$. Denote by $I$ the line  subbundle of $E/H^1$ obtained by  saturating the subsheaf $\theta(H^1)\otimes K_X^{-1}$.
  \begin{description}
    \item[1.1] If $d-2g+2\leq \deg(I)<-d$, then the  Simpson filtration on $(E,\nabla)$ coincides with the Harder-Narasimhan filtration and
    $$\lim\limits_{t\rightarrow 0}t\cdot(E,\nabla)=(H^1\oplus E/H^1, \left(
                                                                       \begin{array}{cc}
                                                                         0 & 0 \\
                                                                         \theta& 0 \\
                                                                       \end{array}
                                                                     \right)
    ).$$
    \item[1.2] If $\deg(I)=-d$, although the Simpson filtration on $(E,\nabla)$ is not unique, the limiting polystable Higgs bundle is given by
$$\lim\limits_{t\rightarrow 0}t\cdot(E,\nabla)=(H^1\oplus I, \left(
                                                               \begin{array}{cc}
                                                                 0 & 0 \\
                                                                 \theta & 0 \\
                                                               \end{array}
                                                             \right)
)\oplus (\frac{E/H^1}{I},0).$$
    \item[1.3] If $-d< \deg(I)\leq -\frac{d}{2}$, one defines $$ F^1:=\Ker (E\rightarrow\frac{E/H^1}{I}),$$
then the  Simpson  filtration  is given by $H^1\subset F^1\subset E$, and
 $$\lim\limits_{t\rightarrow 0}t\cdot(E,\nabla)=(H^1\oplus I\oplus  \frac{E/H^1}{I},\left(
                                                                   \begin{array}{ccc}
                                                                     0 & 0 & 0 \\
                                                                     \theta & 0 & 0 \\
                                                                     0 & \varphi & 0 \\
                                                                   \end{array}
                                                                 \right)
),$$
where   the nonzero morphism  $\varphi: I\rightarrow\frac{E/H^1}{I}\otimes K_X$ is induced by $\nabla: F^1\rightarrow E\otimes K_X$.

  \end{description}
  \item Assume the Harder-Narasimhan filtration on $E$ is given by $G^1\subset E$ with $\rank (G^1)=2, \deg(G^1)=l$, then $0<l\leq\frac{2}{3}(2g-2)$. The flat connection $\nabla$ induces a nonzero morphism
  $\vartheta: G^1 \rightarrow E/G^1\otimes K_X$. Denote by $N$ the line subbundle of $G^1$ obtained by saturating the subsheaf $\Ker(\vartheta)$.
  \begin{description}
    \item[2.1]If $2l-2g+2\leq\deg(N)<0$, then the   Simpson filtration on $(E,\nabla)$ coincides with the Harder-Narasimhan filtration and
    $$\lim\limits_{t\rightarrow 0}t\cdot(E,\nabla)=(G^1\oplus E/G^1, \left(
                                                                       \begin{array}{cc}
                                                                         0 & 0 \\
                                                                        \vartheta& 0 \\
                                                                       \end{array}
                                                                     \right)
    ).$$
    \item[2.2] If $\deg(N)=0$,  the limiting polystable Higgs bundle is given by
$$\lim\limits_{t\rightarrow 0}t\cdot(E,\nabla)=(N,0)\oplus(G^1/N\oplus E/G^1, \left(
                                                               \begin{array}{cc}
                                                                 0 & 0 \\
                                                               \vartheta & 0 \\
                                                               \end{array}
                                                             \right)
).$$
    \item[2.3] If $0<\deg(N)\leq \frac{l}{2}$, then the  Simpson  filtration on $(E,\nabla)$ is given by $N\subset G^1\subset E$, and
    $$\lim\limits_{t\rightarrow 0}t\cdot(E,\nabla)=(N\oplus G^1/N\oplus E/G^1,\left(
                                                                   \begin{array}{ccc}
                                                                     0 & 0 & 0 \\
                                                                     \phi & 0 & 0 \\
                                                                     0 & \vartheta & 0 \\
                                                                   \end{array}
                                                                 \right)
),$$
where the nonzero morphism $\phi: N\rightarrow G^1/N\otimes K_X$ is induced by $\nabla:N\rightarrow G^1\otimes K_X$.
 \end{description}
  \item Assume the Harder-Narasimhan filtration on $E$ is given by $A^1\subset A^2\subset E$ with $\rank (A^i)=i, \deg(A^i)=a_i, i=1,2$, then $0<2a_1-a_2\leq 2g-2, 0<2a_2-a_1\leq 2g-2$. The flat connection $\nabla$ induces  nonzero morphisms $\psi: A^1\rightarrow E/A^1\otimes K_X$ and $\chi: A^2\rightarrow E/A^2\otimes K_X$, then   define $J=\psi(A^1)\otimes K_X^{-1}\subset E/A^1$ and $M=\Ker(\chi)\subset A^2$ viewing as line subbundles by saturating, and define $L^1=\Ker(E\rightarrow\frac{E/A^1}{J})$.
     \begin{description}
    \item[3.1]  When  $-a_1<\deg(J)\leq a_2-a_1$,  the  Simpson  filtration  on $(E,\nabla)$ is given by $A^1\subset L^1\subset E$, and
 $$\lim\limits_{t\rightarrow 0}t\cdot(E,\nabla)=(A^1\oplus J\oplus  \frac{E/A^1}{J},\left(
                                                                   \begin{array}{ccc}
                                                                     0 & 0 & 0 \\
                                                                     \psi & 0 & 0 \\
                                                                     0 & \rho & 0 \\
                                                                   \end{array}
                                                                 \right)
),$$ where   the nonzero morphism  $\rho: J\rightarrow\frac{E/A^1}{J}\otimes K_X$ is induced by $\nabla: L^1\rightarrow E\otimes K_X$. In particular, if $\deg(J)=a_2-a_1$, the Simpson  filtration coincides with the Harder-Narasimhan filtration.
\item[3.2] When $\deg(J)=-a_1$,  the limiting polystable Higgs bundle is given by
$$\lim\limits_{t\rightarrow 0}t\cdot(E,\nabla)=(A^1\oplus J, \left(
                                                               \begin{array}{cc}
                                                                 0 & 0 \\
                                                                 \psi & 0 \\
                                                               \end{array}
                                                             \right)
)\oplus (\frac{E/A^1}{J},0).$$
       \item[3.3]
       When  $a_1-2g+2\leq \deg(J)<-a_1$,
       \begin{description}
         \item [3.3.1]if $a_2-a_1<0$, the Simpson  filtration is given by $A_1\subset E$, and
         $$\lim\limits_{t\rightarrow 0}t\cdot(E,\nabla)=(A^1\oplus E/A^1, \left(
                                                                            \begin{array}{cc}
                                                                              0 & 0 \\
                                                                              \psi & 0 \\
                                                                            \end{array}
                                                                          \right)
         ).$$
         \item [3.3.2] if $a_2-a_1\geq0$, \begin{description}
                                 \item[3.2.2.1] for $2a_2-2g+2\leq\deg(M)<0$, the Simpson filtration is given by $A^2\subset E$, and
                                 $$\lim\limits_{t\rightarrow 0}t\cdot(E,\nabla)=(A^2\oplus E/A^2, \left(
                                                                            \begin{array}{cc}
                                                                              0 & 0 \\
                                                                              \chi & 0 \\
                                                                            \end{array}
                                                                          \right)
         ).$$
         \item[3.3.2.2] for $a_2-a_1\geq0,\deg(M)=0$, the limiting polystable Higgs bundle is given by
         $$\lim\limits_{t\rightarrow 0}t\cdot(E,\nabla)=(M,0)\oplus ( A^2/M\oplus E/A^2,\left(
                                                                                          \begin{array}{cc}
                                                                                            0 & 0 \\
                                                                                            \chi & 0 \\
                                                                                          \end{array}
                                                                                        \right)
         ).$$
                                 \item[3.3.2.3] for $a_2-a_1>0, 0<\deg(M)\leq a_2-a_1$, the Simpson filtration is given by
                                 $M\subset A^2\subset E$, and
                                 $$\lim\limits_{t\rightarrow 0}t\cdot(E,\nabla)=(M\oplus A^2/M\oplus E/A^2,\left(
                                                                   \begin{array}{ccc}
                                                                     0 & 0 & 0 \\
                                                                    \varrho & 0 & 0 \\
                                                                     0 & \chi & 0 \\
                                                                   \end{array}
                                                                 \right)
),$$ where the nonzero morphism $\varrho: M\rightarrow A^2/M$ is induced by $\nabla: M\rightarrow A^2\otimes K_X$. In particular, as $\deg(M)=a_2-a_1$, the underlying vector bundle  of the limiting Higgs bundle  coincides with graded  vector bundle from the Harder-Narasimhan filtration.
                               \end{description}

\end{description}
       \end{description}
\end{enumerate}
\end{example}

The following proposition is an application of the above description of Simpson filtration.

\begin{proposition}
Let $E$ be a vector bundle of rank 3 and degree 0 over $X$, and let $W$ be any subbundle of $E$. If $E$ admits a flat $\lambda$-connection ($\lambda\neq 0$), then    $\deg(W)\leq 4g-4$.
\end{proposition}

\begin{proof}
We only need to consider the case of $\lambda=1$, then we can apply the  the above example.
 Firstly, we consider the case 1.1.
There is a short exact sequence
    $$0\rightarrow W\bigcap H^1\rightarrow W\rightarrow W/(W\bigcap H^1)\rightarrow 0.$$
    Since $ W/(W\bigcap H^1)$ is a subsheaf of $E/H^1$, it follows from the stability of Higgs bundle that
    $$\deg(W)=\deg( W\bigcap H^1)+\deg (W/(W\bigcap H^1))\leq \deg(H^1)=d<g-1.$$
    Next, we consider the cases 1.2 and 1.3. Again by stability of Higgs bundle, we have
    \begin{align*}
     \deg(W/(W\bigcap F^1))&\leq0,\\
     \deg(W\bigcap H^1\oplus I\oplus \frac{E/H^1}{I})&\leq0,\\
     \deg (\frac{W\bigcap F^1}{W\bigcap H^1}\oplus \frac{E/H^1}{I})&\leq0.
    \end{align*}
 Therefore we arrive at
 \begin{align*}
         \deg(W)\leq \deg(W\bigcap F^1)&\leq\deg(W\bigcap H^1)-\deg( \frac{E/H^1}{I})\\
         &\leq2\deg(H^1)+\deg(I)\\
         &\leq\frac{3}{2}d\leq 2g-2.
       \end{align*}
 Similar arguments show that
 $$\deg(W)\leq\left\{
                \begin{array}{ll}
                  l<g-1, & \hbox{case 2.1;} \\
                  \frac{3}{2}l\leq2g-2, & \hbox{cases 2.2 and 2.3;} \\
                  a_1+a_2\leq4g-4, & \hbox{cases 3.1 and 3.2;}\\
                a_1<g-1, & \hbox{case 3.3.1;}\\
                 a_2<g-1, & \hbox{case 3.3.2.1;}\\
                 2a_2-a_1\leq2g-2 & \hbox{cases 3.3.2.2 and 3.3.2.3}.\end{array}
              \right.
 $$
 The conclusion follows.
\end{proof}

For a $\lambda$-flat bundle $(E,D^\lambda) (\lambda\neq0)$, Simpson showed the limit $\lim\limits_{t\to0}t\cdot(E,D^\lambda)$ can be obtained by taking the grading of the Simpson filtration on the associated flat bundle $(E,\lambda^{-1}D^\lambda)$, hence it is a $\mathbb{C}$-VHS, namely, we have

\begin{theorem}[\cite{CS6}]\label{limit}
Let $(E,D^\lambda)\in\mathbb{M}_{\mathrm{Hod}}(X,r)$ be a $\lambda$-flat bundle ($\lambda\neq 0$), then we have
$$
\lim\limits_{t\rightarrow 0}t\cdot(E,D^\lambda)=(\mathrm{Gr}_{\mathcal{F}_\lambda}(E),\mathrm{Gr}_{\mathcal{F}_\lambda}(\lambda^{-1}D^\lambda)),
$$
where ${\mathcal{F}_\lambda}$ is the Simpson filtration on the associated flat bundle $(E,\lambda^{-1}D^\lambda)$.
\end{theorem}

With previous notation, divide $V(X,r)$  into connected components $V(X,r)=\coprod_\alpha V_\alpha$. Define the subset $\mathbb{S}_\alpha\subset \mathbb{M}_{\mathrm{dR}}(X,r)=\pi^{-1}(1)$ via
$$
\mathbb{S}_\alpha:=\left\{y\in\mathbb{M}_{\mathrm{dR}}(X,r)\  \Big|\ \lim\limits_{t\to0}t\cdot y\in V_\alpha\right\},
$$
then each $\mathbb{S}_\alpha$ is locally closed, and these subsets partition the de Rham moduli space into the \emph{oper stratification} \cite{CS6}
$$
\mathbb{M}_{\mathrm{dR}}(X,r)=\coprod_\alpha\mathbb{S}_\alpha,\quad M_{\mathrm{dR}}(X,r)=\coprod_\alpha S_\alpha,
$$
where $S_\alpha:=\mathbb{S}_\alpha\bigcap M_{\mathrm{dR}}(X,r)$.

\begin{definition}
For the oper stratification $M_{\mathrm{dR}}(X,r)=\coprod_\alpha S_\alpha$, a (non-empty) stratum $S_\alpha$ is called
\begin{enumerate}
  \item the \emph{maximal stratum} if it has maximal dimension among all the (non-empty) strata,
  \item the \emph{minimal stratum} if it has minimal dimension  among all the (non-empty) strata.
\end{enumerate}
\end{definition}

\begin{example}
A \emph{uniformizing Higgs bundle} $(E,\theta)$ of rank $r\geq 2$ is given by
$$
E=L\oplus L\otimes K_X^{-1}\oplus\cdots\oplus L\otimes K_X^{-r+1}
$$
and $\theta$ is determined by the natural isomorphisms $L\otimes K_X^i\xrightarrow{\simeq} (L\otimes K_X^{i-1})\otimes K_X$,
where $L$ is a line bundle of degree $(r-1)(g-1)$. The moduli space $V_{\mathrm{uni}}$ of uniformizing Higgs bundles is isomorphic to $\mathrm{Jac}^{(r-1)(g-1)}(X)$, which is an irreducible component of $V(X,r)$. The stratum $S_{\mathrm{oper}}$ corresponding to $V_{\mathrm{uni}}$ is the moduli space of $\mathrm{GL}(r,\mathbb{C})$-opers, called the \emph{oper stratum}. It is shown that  the stratum $S_{\mathrm{oper}}$ is closed in $M_{\mathrm{dR}}(X,r)$ \cite{BB}.
\end{example}

\section{Oper Stratum Conjecture}

Let $u\in V(X,r)$ be a $\mathbb{C}$-VHS, we define the following \emph{Hodge fiber} over $u$ in $\mathbb{M}_{\mathrm{Hod}}(X,r)$ as 
$$
Y_u :=\left\{(E, D^\lambda)\in\mathbb{M}_{\mathrm{Hod}}(X,r)\ \Big|\ \lim\limits_{t\to0}t\cdot(E,D^\lambda)=u\right\},
$$
and define $Y_u^\lambda:=Y_u\bigcap\pi^{-1}(\lambda)$, called the \emph{$\lambda$-Hodge fiber}. It's clear $Y_u^\lambda\simeq Y_u^{\lambda'}$ as analytic varieties whenever $\lambda,\lambda'\in\mathbb{C}^*$. Moreover, when $u$ is stable, then all fibers $Y_u^\lambda, \lambda\in\mathbb{C}$ are analytic isomorphic (\cite{C,H-H}), in particular, these isomorphisms can be realized by conformal limits (\cite{C}).

In the following, we show that each $\lambda$-Hodge fiber $Y_u^\lambda$ over a stable $\mathbb{C}$-VHS $u$ is of half-dimension of the moduli space $\mathbb{M}_{\mathrm{Hod}}^\lambda(X,r)$. This property has been shown by several authors (cf. \cite{CS6} for $\lambda=1$, and \cite{C} for $\lambda=0$). 

\begin{lemma}\label{halfdim}
Let   $u\in V(X,r)$ be a  stable $\mathbb{C}$-VHS, then the tangent space $T_{v}Y^\lambda_u$ at $v\in Y^\lambda_u$ is of half-dimension of $T_v\mathbb{M}_{\mathrm{Hod}}^\lambda(X,r)$. 
\end{lemma}

\begin{proof}
For our purpose, here we just show the case $\lambda\neq0$. 

Write $u=(\mathcal{E},\theta)=(\bigoplus_{i=1}^{k}E_{i},\bigoplus_{i=1}^{k-1}\theta_i:E_i\rightarrow E_{i+1}\otimes K_X)$, and let $v=(E,D^\lambda) \in  Y_u^\lambda$, namely there is a Simpson filtration $\mathcal{F}=\{F^\bullet\}: E=F^0\supset F^1\supset\cdots F^k=0$ on the flat bundle $(E,\lambda^{-1}D^\lambda)$ such that $(\mathrm{Gr}_{\mathcal{F}}(E),\mathrm{Gr}_{\mathcal{F}}(\lambda^{-1}D^\lambda))=u$. The Simpson filtration $\{F^\bullet\}$ produces a filtration  $\widetilde{\mathcal{F}}$ on $\End(E)$ by 
$$
F^p(\End(E))=\{\varphi\in \End(E): \varphi: F^q\rightarrow F^{p+q}\textrm{ for any }q\},
$$ 
hence the graded objects are given by
 $\mathfrak{E}^p:=\mathrm{Gr}_{\widetilde{\mathcal{F}}}^p(\End(E))=\bigoplus_{i=1}^{k}\Hom(E_i,E_{i+p})$.
Then there is  an  filtration on $\End(E)\otimes \Omega^1_X$ by 
$$
F^p(\End(E)\otimes \Omega^1_X)=F^{p-1}(\End(E))\otimes\Omega^1_X,
$$ 
which induces a filtration on the hypercohomology $\mathbb{H}^i(\Omega_X^\bullet(\End(E)),\lambda^{-1}D^{\lambda})$.
By Lemma 7.1 in \cite{CS6}, we have
\begin{align*}
T_{v}Y^\lambda_u\simeq&F^1(\mathbb{H}^1(\Omega_X^\bullet(\End(E)),\lambda^{-1}D^{\lambda}))\\
 \simeq&\bigoplus_{p= 1}^{k-1}\mathbb{H}^1(Gr^p_{\widetilde{\mathcal{F}}}(\End(E))\xrightarrow{\mathrm{Gr}_{\widetilde{\mathcal{F}}}(\lambda^{-1}D^{\lambda})}\mathrm{Gr}^{p-1}_{\widetilde{\mathcal{F}}}(\End(E))\otimes K_X)\\
 \simeq&\left\{(\alpha,\beta)\in\Omega^{0,1}_X(\bigoplus_{p=1}^{k-1}\mathfrak{E}^p)\oplus\Omega^{1,0}_X(\bigoplus_{p=0}^{k-1}\mathfrak{E}^p) :\bar\partial_{\mathcal{E}}\beta+[\theta,\alpha]
  =\partial_{\mathcal{E},h}\alpha+[\theta^\dagger_h,\beta]=0\right\},
  \end{align*}
  where $h$ is the pluri-harmonic metric on $(\mathcal{E},\theta)$.
The dimension of the last space can  be calculated by Riemann-Roch formula as done in Lemma 3.6 of \cite{C}.
\end{proof}

The following properties of $V^{\overrightarrow{r},\overrightarrow{d}}$ are crucial in the proof of our main theorem.

\begin{lemma} [{\cite[Theorem 4.1]{BGG}}, {\cite[Theorem 3.8 (iv)]{A}}]\label{irred}
Given type $\overrightarrow{r}=(r_1,\cdots,r_l),\overrightarrow{d}=(d_1,\cdots,d_{l})$,  we have
\begin{enumerate}
  \item if $V^{\overrightarrow{r},\overrightarrow{d}}$ is not empty, then it is irreducible;
  \item if $V^{\overrightarrow{r},\overrightarrow{d}}$ consists of stable $\mathbb{C}$-VHSs, then its dimension is given by
  $$
  \dim_\mathbb{C}V^{\overrightarrow{r},\overrightarrow{d}}=(g-1)\sum_{i=1}^lr_i(r_i+r_{i+1})+\sum_{i=1}^lr_i(d_{i+1}-d_{i-1})+1,
  $$
where one assigns $r_{l+1}=d_0=d_{l+1}=0$.
\end{enumerate}
\end{lemma}

Applying Lemma \ref{irred}, we can obtain the following positivity of codimension of $V^{\overrightarrow{r},\overrightarrow{d}}_{\mathrm{ss}}$.

\begin{proposition}\label{ff}
Let $V^{\overrightarrow{r},\overrightarrow{d}}_\mathrm{s}\subset V^{\overrightarrow{r},\overrightarrow{d}}$ be the subset consisting of stable $\mathbb{C}$-VHSs, and let $V^{\overrightarrow{r},\overrightarrow{d}}_{\mathrm{ss}}=V^{\overrightarrow{r},\overrightarrow{d}}\backslash V^{\overrightarrow{r},\overrightarrow{d}}_\mathrm{s}$. Assume $V^{\overrightarrow{r},\overrightarrow{d}}_\mathrm{s}$ and $V^{\overrightarrow{r},\overrightarrow{d}}_{\mathrm{ss}}$ are both non-empty, then
then the codimension of $V^{\overrightarrow{r},\overrightarrow{d}}_{\mathrm{ss}}$ in $V^{\overrightarrow{r},\overrightarrow{d}}$ is positive.
\end{proposition}

\begin{proof}
Given a  partition $\{\overrightarrow{r}^{(a)},\overrightarrow{d}^{(a)}\}_{a=1}^k$ with 
\begin{align*}
\sum_{a=1}^k\overrightarrow{r}^{(a)}&=\overrightarrow{r},\quad \sum_{a=1}^k\overrightarrow{d}^{(a)}=\overrightarrow{d},\\
 |\overrightarrow{d}^{(a)}|&=0, \quad a=1,\cdots, k,
\end{align*}
one denotes $V^{\{\overrightarrow{r}^{(a)},\overrightarrow{d}^{(a)}\}}_{\mathrm{ss}}=V^{\{\overrightarrow{r}^{(1)},\overrightarrow{d}^{(1)}\}}_{\mathrm{ss}}\times \cdots \times V^{\{\overrightarrow{r}^{(k)},\overrightarrow{d}^{(k)}\}}_{\mathrm{ss}}$, and defines an injective map
$$
f_{\{\overrightarrow{r}^{(a)},\overrightarrow{d}^{(a)}\}}: V^{\{\overrightarrow{r}^{(a)},\overrightarrow{d}^{(a)}\}}_{\mathrm{ss}}\longrightarrow V^{\overrightarrow{r},\overrightarrow{d}}
$$
by taking direct sum of each component on the left hand side. Since  each point in $V^{\overrightarrow{r},\overrightarrow{d}}$ represents a polystable $\mathbb{C}$-VHS, we have
$$
\dim_\mathbb{C}V_{\mathrm{ss}}^{\overrightarrow{r},\overrightarrow{d}}=\dim_\mathbb{C}\bigcup_{\{\overrightarrow{r}^{(a)},\overrightarrow{d}^{(a)}\}}\mathrm{Im}(f_{\{\overrightarrow{r}^{(a)},\overrightarrow{d}^{(a)}\}})
=\max_{\{\overrightarrow{r}^{(a)},\overrightarrow{d}^{(a)}\}}\dim_\mathbb{C}V^{\{\overrightarrow{r}^{(a)},\overrightarrow{d}^{(a)}\}}_{\mathrm{ss}}.
$$
Since $V^{\overrightarrow{r},\overrightarrow{d}}_{\mathrm{ss}}$ is non-empty and our chains are indecomposable, the set $\{i:r_i>1\}$ is non-empty. We only care about the lower bound of $\mathrm{codim}_\mathbb{C}V^{\overrightarrow{r},\overrightarrow{d}}_{\mathrm{ss}}$, then it is clear that we can take $\overrightarrow{r}=\overrightarrow{r^\prime}+\overrightarrow{r^{\prime\prime}},\overrightarrow{d}=\overrightarrow{d^\prime}+\overrightarrow{d^{\prime\prime}}$ with $\overrightarrow{r^\prime}=(r_1,\cdots,r_{m-1},r_m-1,r_{m+1},\cdots,r_l), \overrightarrow{r^{\prime\prime}}=(0,\cdots,0,1,0,\cdots,0),\overrightarrow{d^\prime}=\overrightarrow{d},\overrightarrow{d^{\prime\prime}}=0$. From  the dimension formula in Lemma \ref{irred} (2)  it follows that
\begin{align*}
\mathrm{codim}_\mathbb{C}V^{\overrightarrow{r},\overrightarrow{d}}_{\mathrm{ss}}&=\dim_\mathbb{C}V^{\overrightarrow{r},\overrightarrow{d}}-\dim_\mathbb{C}\bigg(\max_{\{\overrightarrow{r}^{(a)},\overrightarrow{d}^{(a)}\}}\dim_\mathbb{C}V^{\{\overrightarrow{r}^{(a)},\overrightarrow{d}^{(a)}\}}_{\mathrm{ss}}\bigg)\\
&=(g-1)(2r_m+r_{m+1}+r_{m-1})+d_{m+1}-d_{m-1}+1.
\end{align*}
By Proposition \ref{11}, we have
\begin{align*}
  d_{m-1}-d_m&\leq(2g-2)\min
  \{r_{m-1},r_m\},\\
   d_{m}-d_{m+1}&\leq(2g-2)\min
  \{r_{m},r_{m+1}\},
\end{align*}
then we arrive at
$$\mathrm{codim}_\mathbb{C}V^{\overrightarrow{r},\overrightarrow{d}}_{\mathrm{ss}}\geq\left\{
                                                                                     \begin{array}{ll}
                                                                                      (g-1)(r_{m+1}+r_{m-1}-2r_m)+1 , & \hbox{$r_{m-1}\geq r_m, r_{m+1}\geq r_m$;} \\
                                                                                        (g-1)(2r_m-r_{m+1}-r_{m-1})+1 , & \hbox{$r_{m-1}\leq r_m, r_{m+1}\leq r_m$;}\\
 (g-1)(r_{m+1}-r_{m-1})+1 , & \hbox{$ r_{m+1}\geq r_m\geq r_{m-1}$;} \\
(g-1)(r_{m-1}-r_{m+1})+1 , & \hbox{$ r_{m+1}\leq r_m\leq r_{m-1}$,}
                                                                                     \end{array}
                                                                                   \right.
$$
which implies  $\mathrm{codim}_\mathbb{C}V^{\overrightarrow{r},\overrightarrow{d}}_{\mathrm{ss}}\geq1$.
\end{proof}

Now we give the lower and upper bounds of the dimension of non-empty irreducible components $V^{\overrightarrow{r},\overrightarrow{d}}$, and discuss when these bounds reach.

\begin{theorem} \label{m} 
If $V^{\overrightarrow{r},\overrightarrow{d}}$ is not empty, then we have the inequalities
$$
g\leq\dim_\mathbb{C}V^{\overrightarrow{r},\overrightarrow{d}}\leq r^2(g-1)+1.
$$
 In particular, the equality on the left hand side holds only when 
 $$
 \overrightarrow{r}=(1,\cdots,1),\overrightarrow{d}=((r-1)(g-1),(r-3)(g-1),\cdots,(-r+1)(g-1)),
 $$ 
and  the equality on the right hand side holds only when $\overrightarrow{r}=(r), \overrightarrow{d}=(0)$.
\end{theorem}

\begin{proof}
Firstly we show the inequality $\dim_\mathbb{C}V^{\overrightarrow{r},\overrightarrow{d}}\geq g$, and the equality holds if and only if 
 $
 \overrightarrow{r}=(1,\cdots,1),\overrightarrow{d}=((r-1)(g-1),(r-3)(g-1),\cdots,(-r+1)(g-1)),
 $. 
 
 Consider the case $\overrightarrow{r}=(1,\cdots,1)$, under which, the $\mathbb{C}$-VHSs lying in $V^{\overrightarrow{r},\overrightarrow{d}}$ are stable. By above half-dimension property (Lemma \ref{halfdim}), we have \begin{align*}
\dim_\mathbb{C}V^{\overrightarrow{r},\overrightarrow{d}}&=(g-1)(2r-1)+\sum_{i=1}^{r-1}(d_{i+1}-d_i)+1\\
              &\geq(g-1)(2r-1)-(r-1)(2g-2)+1=g,
\end{align*}
where we have used Proposition \ref{11} for the second inequality. In particular, the equality holds  if and only if $\overrightarrow{d}=((r-1)(g-1),(r-3)(g-1),\cdots,(-r+1)(g-1))$.

For the case when there exists some $r_i>1$,  since stability is an open condition, one can consider our problem at the level of stack. Let $\mathcal{V}^{\overrightarrow{r},\overrightarrow{d}}$ be the corresponding moduli stack, and let $\mathcal{V}^{\overrightarrow{r},\overrightarrow{d}}_{\mathrm{s}}$ be the  substack consisting of stable objects.  It follows from  $\mathcal{V}^{\overrightarrow{r},\overrightarrow{d}}_{\mathrm{s}}$  being a $\mathbb{G}_m$-gerbe over its coarse moduli space $V^{\overrightarrow{r},\overrightarrow{d}}_{\mathrm{s}}$ \cite{PHS} that $\dim_\mathbb{C}V^{\overrightarrow{r},\overrightarrow{d}}\geq\dim_\mathbb{C}\mathcal{V}'+1$, where $\mathcal{V}'$ is a substack of $\mathcal{V}^{\overrightarrow{r},\overrightarrow{d}}$  containing $\mathrm{Bun}^{r_{i_0},d_{i_0}}(X)$ for some $r_{i_0}>1, d_{i_0}$. Here $\mathrm{Bun}^{1,d_{i_0}}(X)\times\mathrm{Bun}^{1,d_{i_1}}(X)$ for some $d_{i_0}, d_{i_1}$, where $\mathrm{Bun}^{r_i,d_i}(X)$ denotes the moduli stack of vector bundles of rank $r_i$ and degree $d_i$ over $X$. Therefore, we have 
$$
\dim_\mathbb{C}V^{\overrightarrow{r},\overrightarrow{d}}\geq\max\big\{r_{i_0}^2(g-1)+1,2g-1\big\}>g.
$$

Now we show the inequality $\dim_\mathbb{C}V^{\overrightarrow{r},\overrightarrow{d}}\leq r^2(g-1)+1$, and the equality holds only when $\overrightarrow{r}=(r),\overrightarrow{d}=(0)$. 

Note that all $V^{\overrightarrow{r},\overrightarrow{d}}$ are contained in the nilpotent cone, so by Lemma \ref{irred}, each $V^{\overrightarrow{r},\overrightarrow{d}}$ is contained in certain irreducible component of the nilpotent cone. It is known that each such irreducible component is a Lagrangian submanifold of $\mathbb{M}_{\mathrm{Dol}}(X,r)$, this is due to Laumon \cite{Lau}, Faltings \cite{Fal}, Beilinson--Drinfeld \cite{BB2}, and Ginzburg \cite{Gin}. So the inequality $\dim_\mathbb{C}V^{\overrightarrow{r},\overrightarrow{d}}\leq r^2(g-1)+1$ holds. On the other hand, by Lemma 11.9 of \cite{CSaa}, the dimension of $V^{\overrightarrow{r},\overrightarrow{d}}$ is strictly less than $r^2(g-1)+1$, the dimension of the nilpotent cone, if $V^{\overrightarrow{r},\overrightarrow{d}}$ contains $\mathbb{C}$-VHSs with non-zero Higgs fields. Therefore, the equality holds only when $V^{\overrightarrow{r},\overrightarrow{d}}=V^{(r),(0)}=i(\mathbb{U}(X,r))$, where $\mathbb{U}(X,r)$ is the moduli space of semistable vector bundles over $X$ of rank $r$ and degree $0$, and $i: \mathbb{U}(X,r)\hookrightarrow\mathbb{M}_{\mathrm{Dol}}(X,r)$ is the natural embedding. 
\end{proof}

\begin{example}
Here we give an explicit description of $V^{\overrightarrow{r},\overrightarrow{d}}$ for the cases $|\overrightarrow{r}|=\sum_{i=1}^lr_i=3, 4$.
\begin{enumerate}
\item When $|\overrightarrow{r}|=\sum_{i=1}^lr_i=3$, then $V^{\overrightarrow{r},\overrightarrow{d}}$ can be described as the following:

Case I: $\overrightarrow{r}=(3),\overrightarrow{d}=(0)$. In this case, $V^{\overrightarrow{r},\overrightarrow{d}}=i(\mathbb{U}(X,3))$,  which is known to be irreducible of dimension $9g-8$.

Case II:  $\overrightarrow{r}=(1,1,1),\overrightarrow{d}=(d_1,d_2,-d_1-d_2)$. We have the isomorphism \cite{P1,G}
\begin{align*}
 V^{\overrightarrow{r},\overrightarrow{d}}&\simeq \mathrm{Jac}^{d_1}(X)\times S^{d_2-d_1+2g-2}X\times S^{-d_1-2d_2+2g-2}X\\
 (E_1,E_2,E_3;\theta_1,\theta_2)&\mapsto(E_1,\mathrm{ div}(\theta_1),\mathrm{div}(\theta_2)),
\end{align*}
where $\mathrm{Jac}^{d_1}(X)$ is the moduli space of line bundles over $X$ of degree $d_1$, $S^pX$ denotes the space of effective divisors of degree $p$ in $X$ which is isomorphic to $X^p/S_p$ for the symmetry group $S_p$, and $\mathrm{ div}(\theta_i)$ stands for the effective divisor of the morphism $\theta_i$. 

Case III: $\overrightarrow{r}=(1,2),\overrightarrow{d}=(d_1,-d_1)$. If $0<d_1<g-1$, then $V^{\overrightarrow{r},\overrightarrow{d}}_{\mathrm{s}}$  is non-empty, and  it is birationally equivalent to a $\mathbb{P}^N$-fibration over $\mathrm{Jac}^{-2d_1+2g-2}(X)\times \mathrm{Jac}^{d_1-2g+2}(X)$ for $N=-3d_1+5g-6$ \cite{B,P1}. If $d_1=g-1$, then $V^{\overrightarrow{r},\overrightarrow{d}}_{\mathrm{s}}$  is empty, and $V^{\overrightarrow{r},\overrightarrow{d}}\simeq \mathrm{Jac}^{0}(X)\times\mathrm{Jac}^{-g+1}(X)$ \cite{B}.

Case IV: $\overrightarrow{r}=(2,1),\overrightarrow{d}=(d_1,-d_1)$. If $0<d_1<g-1$, then $V^{\overrightarrow{r},\overrightarrow{d}}_{\mathrm{s}}$  is non-empty, and it is birationally equivalent to a $\mathbb{P}^N$-fibration over $\mathrm{Jac}^{2d_1-2(2g-2)}(X)\times \mathrm{Jac}^{-d_1}(X)$ \cite{B,P1}. If $d_1=g-1$, then $V^{\overrightarrow{r},\overrightarrow{d}}_{\mathrm{s}}$  is empty, and $V^{\overrightarrow{r},\overrightarrow{d}}\simeq \mathrm{Jac}^{-2g+2}(X)\times\mathrm{Jac}^{-g+1}(X)$ \cite{B}.

Therefore, we have the dimension formula
$$
\dim_\mathbb{C}V^{\overrightarrow{r},\overrightarrow{d}}
      =\left\{
            \begin{array}{ll}
               9g-8, &  \hbox{$\overrightarrow{r}=(3), \overrightarrow{d}=(0) $;}\\
               5g-4-2d_1-d_2, & \hbox{$\overrightarrow{r}=(1,1,1),\overrightarrow{d}=(d_1,d_2,-d_1-d_2)$;} \\
               7g-6-3d_1& \hbox{$\overrightarrow{r}=(1,2)$ or $(2,1)$,$\overrightarrow{d}=(d_1,-d_1)$, $0<d_1<g-1$;}\\
               2g, & \hbox{$\overrightarrow{r}=(1,2)$ or $(2,1)$,$\overrightarrow{d}=(g-1,-g+1)$}.
           \end{array}
   \right.
$$

Then it follows  from Proposition \ref{11} that we have the inequalities
$$0<2d_1+d_2\leq\left\{
              \begin{array}{ll}
                4g-4, & \hbox{$\overrightarrow{r}=(1,1,1),\overrightarrow{d}=(d_1,d_2,-d_1-d_2)$;} \\
                g-1, & \hbox{$\overrightarrow{r}=(1,2)$ or $(2,1)$,$\overrightarrow{d}=(d_1,d_2=-d_1)$,}
              \end{array}
            \right.
$$
which lead to the  inequalities on $\dim_\mathbb{C}V^{\overrightarrow{r},\overrightarrow{d}}$ as
$$
\left\{
 \begin{array}{ll}
 \qquad \qquad\ \dim_\mathbb{C}V^{\overrightarrow{r},\overrightarrow{d}}=9g-8, \ \ &  \hbox{$\overrightarrow{r}=(3),\overrightarrow{d}=(0) $;}\\
 \quad\quad\ g\leq \dim_\mathbb{C}V^{\overrightarrow{r},\overrightarrow{d}}<5g-4, \ \ & \hbox{$\overrightarrow{r}=(1,1,1),\overrightarrow{d}=(d_1,d_2,-d_1-d_2)$;} \\
4g-3\leq \dim_\mathbb{C}V^{\overrightarrow{r},\overrightarrow{d}}<7g-6, \ \ & \hbox{$\overrightarrow{r}=(1,2)$ or $(2,1)$,$\overrightarrow{d}=(d_1,-d_1)$.}                                                                                       
\end{array}
\right.
$$ 

\item Similarly when $|\overrightarrow{r}|=\sum_{i=1}^lr_i=4$, then $V^{\overrightarrow{r},\overrightarrow{d}}$ can be described as the following:

Case I: $\overrightarrow{r}=(4),\overrightarrow{d}=(0)$. In this case, $V^{\overrightarrow{r},\overrightarrow{d}}=i(\mathbb{U}(X,4))$, which is known to be irreducible of dimension $16g-15$.

Case II:  $\overrightarrow{r}=(1,1,1,1),\overrightarrow{d}=(d_1,d_2,d_3,d_4=-d_1-d_2-d_3)$. We have the isomorphism
\begin{align*}
 V^{\overrightarrow{r},\overrightarrow{d}}&\simeq \mathrm{Jac}^{d_1}(X)\times S^{d_2-d_1+2g-2}X\times S^{d_3-d_2+2g-2}X\times S^{-d_1-d_2-2d_3+2g-2}X\\
 (E_1,E_2,E_3,E_4;\theta_1,\theta_2,\theta_3)&\mapsto(\mathcal{E}_1,\mathrm{ div}(\theta_1),\mathrm{div}(\theta_2),\mathrm{div}(\theta_3)).
\end{align*}

Case III: $\overrightarrow{r}=(1,3),\overrightarrow{d}=(d_1,-d_1)$.   If $0<d_1<g-1$, then $V^{\overrightarrow{r},\overrightarrow{d}}_{\mathrm{s}}$  is non-empty, and it is birationally equivalent to a $\mathbb{P}^N$-fibration over $U(X,2,-2d_1+2g-2)\times\mathrm{ Jac}^{d_1-2g+2}(X)$ for $N=-4d_1+8g-9$  \cite{B}, where $U(X,r,d)$ denotes the moduli space of stable vector bundles over $X$ of rank $r$ and degree $d$.  If $d_1=g-1$, then $V^{\overrightarrow{r},\overrightarrow{d}}_{\mathrm{s}}$  is empty, and
$V^{\overrightarrow{r},\overrightarrow{d}}\simeq \mathbb{U}(X,2)\times\mathrm{ Jac}^{-g+1}(X)$ \cite{B}.

Case IV: $\overrightarrow{r}=(3,1),\overrightarrow{d}=(d_1,-d_1)$. If $0<d_1<g-1$, then $V^{\overrightarrow{r},\overrightarrow{d}}_{\mathrm{s}}$  is non-empty, and it is birationally equivalent to a $\mathbb{P}^N$-fibration over $U(X,2,2d_1-6g+6)\times \mathrm{Jac}^{-d_1}(X)$  \cite{B}. If $d_1=g-1$, then $V^{\overrightarrow{r},\overrightarrow{d}}_{\mathrm{s}}$  is empty, and
$V^{\overrightarrow{r},\overrightarrow{d}}\simeq \mathbb{U}(X,2,-4g+4)\times \mathrm{Jac}^{-g+1}(X)$ \cite{B}, where $\mathbb{U}(X,r,d)$ denotes the moduli space of semistable vector bundles over $X$ of rank $r$ and degree $d$.

Case V: $\overrightarrow{r}=(2,2),\overrightarrow{d}=(d_1,-d_1)$. In this case, $V^{\overrightarrow{r},\overrightarrow{d}}_{\mathrm{s}}$ is birationally equivalent to a $\mathbb{P}^{N'}$-fibration over $U(X,2,d_1-4g+4)\times S^{-2d_1+4g-4}X$ for $N'=-2d_1+4g-4$ \cite{B}.

Case VI: $\overrightarrow{r}=(1,1,2),\overrightarrow{d}=(d_1,d_2,-d_1-d_2)$. If $d_1-d_2\leq 2g-2, 2d_2+d_1<2g-2$, then $V^{\overrightarrow{r},\overrightarrow{d}}_{\mathrm{s}}$ is non-empty, and  it is birationally equivalent to a $\mathbb{P}^{N''}$-fibration over $\mathrm{ Jac}^{d_2-2g+2}(X)\times \mathrm{ Jac}^{-d_1-2d_2+2g-2}(X)\times S^{d_2-d_1+2g-2}X$ for $N''=-d_1-3d_2+5g-6$ \cite{A}. If  $d_1-d_2\leq 2g-2, 2d_2+d_1=2g-2$, then $V^{\overrightarrow{r},\overrightarrow{d}}_{\mathrm{s}}$ is empty, and  $V^{\overrightarrow{r},\overrightarrow{d}}\simeq \mathrm{ Jac}^{d_2-2g+2}(X)\times \mathrm{ Jac}^{0}(X)\times S^{d_2-d_1+2g-2}X$.

Case VII: $\overrightarrow{r}=(2,1,1),\overrightarrow{d}=(d_1,d_2,-d_1-d_2)$. If $d_1-d_2<2g-2, 2d_2+d_1\leq 2g-2$, then $V^{\overrightarrow{r},\overrightarrow{d}}_{\mathrm{s}}$ is non-empty, and it is birationally equivalent to a smooth irreducible  variety of dimension $9g-8-2d_1$ \cite{A}. If $d_1-d_2=2g-2, 2d_2+d_1\leq 2g-2$, then $V^{\overrightarrow{r},\overrightarrow{d}}_{\mathrm{s}}$ is empty, and $V^{\overrightarrow{r},\overrightarrow{d}}\simeq\mathrm{ Jac}^0(X)\times \mathrm{ Jac}^{d_1-4g+4}(X)\times S^{-3d_1+6g-6}(X)$.

Case VIII: $\overrightarrow{r}=(1,2,1),\overrightarrow{d}=(d_1,d_2,-d_1-d_2)$.  If $d_1-d_2\leq 2g-2, 2d_2+d_1\leq 2g-2, \overrightarrow{d}\neq (2g-2,0,-2g+2)$, then $V^{\overrightarrow{r},\overrightarrow{d}}_{\mathrm{s}}$  is non-empty, and it is birationally equivalent to a $\mathbb{P}^{N'''}$-fibration over $\mathrm{ Jac}^{d_1-4g+4}(X)\times \mathrm{ Jac}^{d_2-d_1}(X)\times S^{-2d_1-d_2+4g-4}X$ for $ N^{\prime\prime\prime}=-2d_1-d_2+4g-5$ \cite{A}. If  $ \overrightarrow{d}=(2g-2,0,-2g+2)$, then $V^{\overrightarrow{r},\overrightarrow{d}}_{\mathrm{s}}$  is empty, and $V^{\overrightarrow{r},\overrightarrow{d}}\simeq \mathrm{ Jac}^{-2g+2}(X)\times \mathrm{ Jac}^{-g+1}(X)$.

Therefore, the dimension of each irreducible component is given by
$$
\dim_\mathbb{C}V^{\overrightarrow{r},\overrightarrow{d}}=
       \left\{
          \begin{array}{ll}
            16g-15, &  \hbox{$\overrightarrow{r}=(4),\overrightarrow{d}=(0) $;}\\
            7g-6-2d_1-d_2-d_3, & \hbox{$\overrightarrow{r}=(1,1,1,1),\overrightarrow{d}=(d_1,d_2,d_3,-d_1-d_2-d_3)$;} \\
           13g-12-4d_1, & \hbox{$\overrightarrow{r}=(1,3)$ or $(3,1)$,$\overrightarrow{d}=(d_1,-d_1),d_1\neq g-1$;}\\
           5g-3, & \hbox{$\overrightarrow{r}=(1,3)$ or $(3,1)$,$\overrightarrow{d}=(g-1,-g+1)$;}\\                                                                                       
           12g-11-4d_1, & \hbox{$\overrightarrow{r}=(2,2),\overrightarrow{d}=(d_1,-d_1)$;}\\
           9g-8-2d_1-2d_2, & \hbox{$\overrightarrow{r}=(1,1,2)$,$\overrightarrow{d}=(d_1,d_2,-d_1-d_2),2d_2+d_1\neq 2g-2$;}\\
          5g-3-\frac{3}{2}d_1, & \hbox{$\overrightarrow{r}=(1,1,2)$,$\overrightarrow{d}=(d_1,g-1-\frac{d_1}{2},-g+1-\frac{d_1}{2})$;}\\
          9g-8-2d_1, & \hbox{$\overrightarrow{r}=(2,1,1)$,$\overrightarrow{d}=(d_1,d_2,-d_1-d_2),d_1-d_2\neq 2g-2$;}\\
          8g-6-3d_1, & \hbox{$\overrightarrow{r}=(2,1,1)$,$\overrightarrow{d}=(d_1,-2g+2+d_1,2g-2-2d_1)$;}\\
         10g-9-4d_1-2d_2, &  \hbox{$\overrightarrow{r}=(1,2,1)$, $\overrightarrow{d}=(d_1,d_2,-d_1-d_2)\neq(2g-2,0,-2g+2)$;}   \\ 2g, &  \hbox{$\overrightarrow{r}=(1,2,1)$, $\overrightarrow{d}=(2g-2,0,-2g+2)$.}
       \end{array}
        \right.
$$

Then it again follows  from Proposition \ref{11} that we have the inequalities
\begin{align*}
 0&<2d_1+d_2+d_3\leq\left\{
              \begin{array}{ll}
                6g-6, & \hbox{$\overrightarrow{r}=(1,1,1,1),\overrightarrow{d}=(d_1,d_2,d_3,-d_1-d_2-d_3)$;} \\
                g-1, & \hbox{$\overrightarrow{r}=(1,3)$ or $(3,1)$,$\overrightarrow{d}=(d_1,d_2=-d_1)$,$d_3=0$;}\\
                2g-2, & \hbox{$\overrightarrow{r}=(2,2)$,$\overrightarrow{d}=(d_1,d_2=-d_1)$,$d_3=0$;}\\
2g-2, & \hbox{$\overrightarrow{r}=(2,1,1)$,$\overrightarrow{d}=(d_1,d_2,-d_1-d_2)$;}
              \end{array}
            \right. \\
            0&<d_1+d_2\leq2g-2,    \overrightarrow{r}=(1,1,2),\overrightarrow{d}=(d_1,d_2,-d_1-d_2); \\
0&<2d_1+d_2\leq 4g-4, \overrightarrow{r}=(1,2,1), \overrightarrow{d}=(d_1,d_2,-d_1-d_2).
   \end{align*}
Hence, we arrive at
$$
\left\{
\begin{array}{ll}
\qquad\qquad \dim_\mathbb{C}V^{\overrightarrow{r},\overrightarrow{d}}=16g-15, \ \ &  \hbox{$\overrightarrow{r}=(4),\overrightarrow{d}=(0) $;}\\
\quad\quad\  g\leq\dim_\mathbb{C}V^{\overrightarrow{r},\overrightarrow{d}}<7g-6, \ \ & \hbox{$\overrightarrow{r}=(1,1,1,1),\overrightarrow{d}=(d_1,d_2,d_3,-d_1-d_2-d_3)$;} \\
5g-3\leq\dim_\mathbb{C}V^{\overrightarrow{r},\overrightarrow{d}}<13g-12, \ \ & \hbox{$\overrightarrow{r}=(1,3)$ or $(3,1)$,$\overrightarrow{d}=(d_1,-d_1)$;}\\
5g-4\leq\dim_\mathbb{C}V^{\overrightarrow{r},\overrightarrow{d}}<12g-11, \ \ & \hbox{$\overrightarrow{r}=(2,2),\overrightarrow{d}=(d_1,-d_1)$;}\\
\quad\ \ 2g\leq\dim_\mathbb{C}V^{\overrightarrow{r},\overrightarrow{d}}\leq9g-8, \ \ & \hbox{$\overrightarrow{r}=(2,1,1)$ or $(1,1,2)$,$\overrightarrow{d}=(d_1,d_2,-d_1-d_2)$;}\\
  2g-1\leq\dim_\mathbb{C}V^{\overrightarrow{r},\overrightarrow{d}}<10g-9, \ \ &  \hbox{$\overrightarrow{r}=(1,2,1)$, $\overrightarrow{d}=(d_1,d_2,-d_1-d_2)$.}                                                                                       \end{array}
\right.
$$ 
\end{enumerate}
\end{example}

\bigskip

Finally, we complete the proof of our main result.

\begin{corollary} \label{mainthm}
For the oper stratification $M_{\mathrm{dR}}=\coprod_\alpha S_\alpha$, we have
\begin{enumerate}
  \item the open dense stratum $N(X,r)$ consisting of irreducible flat bundles such that the underlying vector bundles are stable is the unique maximal stratum with dimension $2r^2(g-1)+2$,
  \item the closed oper stratum $S_{\mathrm{oper}}$ is the unique minimal stratum with dimension $r^2(g-1)+g+1$.
\end{enumerate}
\end{corollary}

\begin{proof}
We have seen that $S_{\mathrm{oper}}$ is a fibration over $V_{\mathrm{uni}}$ with fibers as  Lagrangian submanifolds of $M_{\mathrm{dR}}(X,r)$, so the dimension is calculated as follows
\begin{align*}
 \dim_\mathbb{C}S_{\mathrm{oper}}&=\dim_\mathbb{C}\mathrm{Jac}^{(r-1)(g-1)}(X)+\frac{1}{2}\dim_\mathbb{C}M_{\mathrm{dR}}(X,r)\\
 &=r^2(g-1)+g+1.
\end{align*}

It follows from Theorem \ref{limit} that if $u$ is a stable $\mathbb{C}$-VHS, then $Y_u^1\subset M_{\mathrm{dR}}(X,r)$, and if $v\in M_{\mathrm{dR}}(X,r)$, then $\lim\limits_{t\to0}t\cdot v$ is a stable $\mathbb{C}$-VHS. For a general stratum $S_\alpha$, the dimension is given by $\dim_{\mathbb{C}}V_\alpha+\frac{1}{2}\dim_\mathbb{C}M_{\mathrm{dR}}(X,r)$, which reaches the maximal or minimum value only when the dimension of $V_\alpha$ reaches the maximal or minimum value. Note that  in two extreme cases when $(\overrightarrow{r},\overrightarrow{d})=((1,\cdots,1),((r-1)(g-1),(r-3)(g-1),\cdots,(-r+1)(g-1))$  and $(\overrightarrow{r},\overrightarrow{d})=((r),(0))$,  $V^{\overrightarrow{r},\overrightarrow{d}}$ are both the connected components of $V(X,r)$, then by Lemma \ref{irred} and Theorem \ref{m} the conclusion follows. 
\end{proof}

\end{document}